\documentclass[a4paper, 11pt]{amsart}
\usepackage[utf8]{inputenc}
\usepackage{amsfonts,amssymb}
\usepackage{graphicx,tikz}
\usepackage{tikz-cd} 
\usepackage{float}
\usepackage{amsmath, amsthm, amsfonts}
\usepackage{hyperref}
\usepackage{enumitem}  
\usepackage[capitalise]{cleveref}
\usepackage{comment}
\usepackage{enumitem}
\usepackage{amsrefs}
\usepackage{url}
\usepackage{nicefrac}
\usepackage{graphicx}
\usepackage{nomencl}
\usepackage{amsfonts}
\usepackage{hyperref}
\usepackage{amssymb}
\usepackage{fancyhdr}
\usepackage{amscd}
\usepackage[english]{babel}
\usepackage{graphicx} 
\usepackage{mathtools}
\usepackage{ marvosym }
\usepackage{amssymb,amsthm,amsmath,amsfonts,latexsym,wasysym}
\usepackage{soul}
\usepackage{enumerate}
\usepackage[all]{xy}
\usepackage{xcolor}

\newtheorem{theorem}{Theorem}[section]
\newtheorem{lemma}[theorem]{Lemma}
\newtheorem{proposition}[theorem]{Proposition}
\newtheorem{corollary}[theorem]{Corollary}
\newtheorem*{thmA}{Theorem A}
\newtheorem*{thmB}{Theorem B}
\theoremstyle{remark}

\newcommand{\N}{\mathbb{N}}
\newcommand{\Z}{\mathbb{Z}}
\newcommand{\Zen}{\mathrm{Z}}
\newcommand{\FC}{\mathrm{FC}}
\newcommand{\BFC}{\mathrm{BFC}}
\newcommand{\tim}{\mathrm{\; times}}


\begin{document}

\title{A generalization of concise words}

\author[C. Delizia, M. Gaeta, C. Monetta]{Costantino Delizia, Michele Gaeta, Carmine Monetta}
\address{Dipartimento di Matematica, Universit\`a di Salerno, Fisciano 84084 (SA), Italy}
\email{cdelizia@unisa.it, migaeta@unisa.it, cmonetta@unisa.it}

\maketitle

\begin{abstract}
The study of verbal subgroups within a group is well-known for being an effective tool to obtain structural information about a group. Therefore, conditions that allow the classification of words in a free group are of paramount importance. One of the most studied problems is to establish which words are concise, where a word
$w$ is said to be concise if 
the verbal subgroup $w(G)$ is finite in each group $G$
in which $w$ takes only a finite number of values.
  The purpose of this article is to present some results, in which a hierarchy among words is introduced, generalizing the concept of concise word.
\end{abstract}

\section{Introduction}

Let $w=w(x_1,\dots, x_n)$ be a group-word in the variables $x_1, \dots, x_n$. For any group $G$ and arbitrary $g_1, \dots, g_n$, the elements $w(g_1, \dots, g_n)$ are called the $w$-values in $G$. We write $G_w$ to denote the set of all $w$-values in $G$ and $w(G)$ to denote the verbal subgroup of $G$ corresponding to $w$, which is the subgroup generated by all $w$-values.
For a subset $S$ of a group $G$ we denote $S^*=S\cup S^{-1}$. Note that if $S$ is a normal set then $S^*$ is normal too.
Clearly, any conjugate of a $w$-value is again a $w$-value, and so $G_w$ and $G_w^*$ are normal sets.

A word $w$ is called concise if the verbal subgroup $w(G)$ is finite in each group $G$
such that $G_w$ is finite. One of the most studied problems is to establish which words are concise. In the sixties, P. Hall conjectured that every word is concise,
but his conjecture was refuted in 1989 by S. Ivanov \cite{I}. However,
many words of common use are known to be concise (see, e.g., \cite{DSTT}, \cite{FMT}).

A word is called semiconcise if the subgroup $[w(G), G]$ is finite in each group $G$
such that $G_w$ is finite. Of course concise words are semiconcise. In \cite{DST} they proved that if $w$ is a semiconcise word and $z$ is any variable not appearing in $w$, then the word $[w,z]$ is also semiconcise. It is not known if there exists a semiconcise word which is not concise, however in \cite{DST} they proved that there exist a word which is not semiconcise.

We consider a generalization of semiconcise words.
Let $w$ be a group-word. The word $w$ is $\frac{1}{n}$-concise if the finiteness of $G_w$ for any group $G$ implies that the subgroup
$$
[w(G), \underbrace{G, \dots, G}_{n-1 \tim}]
$$
is finite.
The word $w$ is $0$-concise if the finiteness of $G_w$ for any group $G$ implies that there exist $n \in \N$ (depending on the group $G$) such that the subgroup
$$
[w(G), \underbrace{G, \dots, G}_{n-1 \tim}]
$$
is finite.
Obviously every $\frac{1}{n}$-concise word is a $\frac{1}{m}$-concise word for every $m \geq n$, and every $\frac{1}{n}$-concise word is a $0$-concise word, so this introduces what we can call a hierarchy on words. It is unknown if there are $\frac{1}{m}$-concise word that are not $\frac{1}{n}$-concise for some $m > n$. However, in a similar way as in \cite{DST}, we can see that there exist a word which is not $\frac{1}{n}$-concise for every $n \in \N$. 

Further information on the subgroup $$
[w(G), \underbrace{G, \dots, G}_{n-1 \tim}]
$$ when $w$ is $\frac{1}{n}$-concise can be obtained using a verbal generalization of $\FC$-groups. For subsets $X$ and $Y$ of a group $G$ we write $X^Y$ to denote the set of conjugates $\{x^y | x\in X, y \in Y\}$.
    Let $G$ be a group and let $H$ be a subgroup of $G$. The subgroup $H$ is said to be $\FC$-embedded in $G$ if $x^H$ is finite for all $x \in G$.  The subgroup $H$ is said to be $\BFC$-embedded in $G$ if $x^H$ is finite for all $x \in G$ and the number of elements in $x^H$ is bounded by a constant that does not depend on the choice of $x$.
 Let $w$ be a group-word and let $G$ be a group. The group $G$ is an $\FC (w)$-group if the set of conjugates $x^{G_w}$ is finite for all $x \in G$. The group $G$ is a $\BFC (w)$-group if $x^{G_w}$ is finite for all $x \in G$ and the number of elements in $x^{G_w}$ is bounded by a constant that does not depend on the choice of $x$.

In this paper we generalize the results in \cite{DST}. More precisely, we prove the following theorems.

\begin{thmA}
  Let $w$ be a $\left(\frac{1}{m+1}\right)$-concise word, with $m \in \N$ and let $G$ be an $\FC (w)$-group. Then 
  $$
    [w(G), \underbrace{G, \dots, G}_{m \tim}]
  $$
  is $\FC$-embedded in $G$.
\end{thmA}
\begin{thmB}
  Let $w$ be a $\left(\frac{1}{m+1}\right)$-concise word, with $m \in \N$ and let $G$ be an $\BFC (w)$-group. Then 
  $$
    [w(G), \underbrace{G, \dots, G}_{m \tim}]
  $$
  is $\BFC$-embedded in $G$.
\end{thmB}

We also show that, for a certain word $w$, there is an
example of a $\BFC(w)$-group $G$ such that $$
    [w(G), \underbrace{G, \dots, G}_{n-1 \tim}]
  $$
  is not $\FC$-embedded in $G$.
Additionally, as a consequence, we have an example of a word which is not $0$-concise. 

\section{Unbounded case}

\begin{lemma}\label{lemma 1}
    Let $w=w(x_1, \dots, x_n)$ be a group-word and set 
    $$
    v=[w(x_1,\dots,x_n), x_{n+1}, \dots, x_{n+m}].
    $$
    Let $y \in G_v$. Then, there exist $w_1, \dots, w_k \in G_w^*$ such that $y=w_1\cdots w_k$, with $k \leq 2^m$.
    \begin{proof}
We argue by induction on $m$. The case $m=1$ was proven in \cite[Lemma 2.3]{DST}.


Assume that the assertion is true in case $m \in \N$. Consider 
\begin{eqnarray*}
    y&=&[w(g_1,\dots,g_n), g_{n+1}, \dots, g_{n+m}, g_{n+m+1}] \\
    &=&\left[[w(g_1,\dots,g_n), g_{n+1}, \dots, g_{n+m}], g_{n+m+1}\right].
\end{eqnarray*}
By induction hypothesis there exist $w_1, \dots, w_k \in G_w^*$ such that $y=[w_1 \cdots w_k, g_{n+m+1}]$, with $k\leq 2^m$. So we have 
\begin{eqnarray*}
    y&=&[w_1 \cdots w_k, g_{n+m+1}] \\
    &=& (w_1 \cdots w_k)^{-1} (w_1 \cdots w_k)^{g_{n+m+1}} \\
    &=& w_k^{-1} \cdots w_1^{-1} w_1^{g_{n+m+1}}\cdots w_k^{g_{n+m+1}}.
\end{eqnarray*}
Note that $w_i^{j} \in G_w^*$, with $i\in\{1, \dots, k\}, j\in\{-1, g_{n+m+1}\}$. So there exist $w'_1, \dots, w'_{k'} \in G_w^*$ such that $y=w'_1 \cdots w'_{k'}$, with $k'\leq 2k \leq 2^{m+1}$. 
    \end{proof}
\end{lemma}

\begin{lemma}\label{lemma 2}
    Let $w=w(x_1, \dots, x_n)$ be a group-word and set 
    $$
    v=[w(x_1,\dots,x_n), x_{n+1}, \dots, x_{n+m}].
    $$
    If $G$ is an $\FC (w)$-group, then it is an $\FC (v)$-group.

\end{lemma}
\begin{proof}
Let $x \in G$ and $y \in G_v$. By Lemma \ref{lemma 1} there exist $w_1, \dots, w_k \in G_w^*$ such that $y=w_1 \cdots w_k$, with $k\leq 2^m$. By \cite[Proposition 2.9(i)]{DST17} $G$ is an $\FC(w^{-1})$-group. Note that $G_{w^{-1}}={G_w}^{-1}$, so it follows that $x^{G_w^*}$ is finite. Let $x^{G_w^*}=\{x^{b_1}, \dots, x^{b_s}\}$ and put $A=\{b_1, \dots, b_s\}$. By \cite[Lemma 2.2]{DST} we have that $x^y=x^{w_1 \cdots w_k}=x^{a_1 \cdots a_k}$, with $a_1, \dots, a_k \in A$. Therefore $|x^{G_v}|\leq |A|^{2^m}$ and so it is finite.
\end{proof}

\begin{lemma}\label{lemma 3}
    Let $w=w(x_1, \dots, x_n)$ be a $\left(\frac{1}{m+1}\right)$-concise word, with $m \in \N$, and set 
    $$
    v=[w(x_1,\dots,x_n), x_{n+1}, \dots, x_{n+m}].
    $$
    Let $G$ be an $\FC (w)$-group and $B$ a finite subset of $G_v^*$. Then, for any $x \in G$, there exists a positive integer $e$ such that $b^e \in \Zen (\langle x, B \rangle)$ for all $b \in B$.
\end{lemma}
\begin{proof}
    Let $B=\{b_1, \dots, b_r\}$ and $x \in G$. For any $b_i \in B$ there exist elements $g_{i_1}, \dots, g_{i_{n+m}} \in G$ such that 
    $$
    b_i=[w(g_{i_1},\dots,g_{i_n}), g_{i_{n+1}}, \dots, g_{i_{n+m}}]^{\varepsilon_i},
   $$
   with $\varepsilon_i \in \{ 1, -1\}$.
    Put
    $$
    J=\langle x, g_{i_j} | 1\leq i \leq r, 1\leq j \leq n+m \rangle.
    $$
    By \cite[Lemma 2.7(i)]{DST17} the set $(J/\Zen(J))_w$ is finite. As $w$ is $(\frac{1}{m+1})$-concise, the subgroup 
    $$
    [w(J/\Zen(J)), \underbrace{J/\Zen(J), \dots, J/\Zen(J)}_{m \tim}]=[w(J), \underbrace{J, \dots, J}_{m \tim} ]\Zen(J)/\Zen(J)
    $$
    is finite. Thus $v(J)$ has finite order modulo $\Zen(J)$, say $e$. Since $B \subseteq v(J)$, it follows that $b_i^{e} \in \Zen(J)$ for all $i$. As $\langle x, B\rangle \leq J$, the result follows.
\end{proof}

\begin{thmA}\label{teorema A}
  Let $w$ be a $\left(\frac{1}{m+1}\right)$-concise word, with $m \in \N$ and let $G$ be an $\FC (w)$-group. Then 
  $$
    [w(G), \underbrace{G, \dots, G}_{m \tim}]
  $$
  is $\FC$-embedded in $G$.
\end{thmA}
\begin{proof}
    Set $v=[w(x_1,\dots,x_n), x_{n+1}, \dots, x_{n+m}]$. Then $G$ is an $\FC (v)$-group by Lemma \ref{lemma 2}. Let $x \in G$. By \cite[Lemma 2.1]{DST}, we can choose $b_1, \dots, b_r \in G_v^*$ such that $x^{G_v^*}=\{x^{b_1}, \dots, x^{b_r}\}$. Write $B=\{b_1, \dots, b_r\}$. Define the order $<$ on the set of all (formal) products of the form $b_{i_1}\cdots b_{i_j}$, with $1\leq i_k \leq r$ and $j \geq 1$, as follows. Put
    $$
    b_{i_1}\cdots b_{i_j} < b_{i'_1}\cdots b_{i'_{j'}}
    $$
    if and only if one of the following conditions is satisfied: $j < j'$ or $j=j'$ and there is a positive integer $l \leq j$ such that $i_l < i'_l$ and $i_k=i'_k$ for all $k > l$. Let $y$ be an arbitrary element of $v(G)$. Then $y=y_1 \cdots y_j$, where each $y_i \in G_v^*$. By \cite[Lemma 2.2]{DST} for all $k \in \{1, \dots, j\}$, there exist an integer $i_k \in \{1, \dots, r\}$ such that $x^y=x^{b_{i_1}\cdots b_{i_j}}$. Clearly, we can choose $b_{i_1} \cdots b_{i_j}$ to be the smallest (respect to $<$) product of elements from $B$ such that $x^y=x^{b_{i_1}\cdots b_{i_j}}$. Now we show that $i_1 \geq i_2 \geq \dots \geq i_j$. Suppose to the contrary that $i_k < i_{k+1}$ for some $k$. Then
    \begin{eqnarray*}
    x^y&=&x^{b_{i_1}\cdots b_{i_{k-1}}b_{i_k}b_{i_{k+1}}b_{i_{k+2}}\cdots b_{i_j}} \\
    &=&x^{b_{i_1}\cdots b_{i_{k-1}}cb_{i_k}b_{i_{k+2}}\cdots b_{i_j}},
    \end{eqnarray*}
    where $c=b_{i_k}b_{i_{k+1}}b_{i_k}^{-1} \in G_v^*$. In view of \cite[Lemma 2.2]{DST}, we have 
    $$
    x^{b_{i_1}\cdots b_{i_{k-1}}c}=x^{b_{i'_1}\cdots b_{i'_{k-1}}b_{i'_{k+1}}}
    $$
    for some $1 \leq i'_1, \dots, i'_{k-1}, i'_{k+1} \leq r$ so that 
    $$
    x^y=x^{b_{i'_1}\cdots b_{i'_{k-1}}b_{i'_{k+1}}b_{i_k}b_{i_{k+2}}\cdots b_{i_j}}.
    $$
This contradicts the choice of the product $b_{i_1}\cdots b_{i_j}$ because $$
b_{i_1}\cdots b_{i_{k-1}}b_{i_k}b_{i_{k+1}}b_{i_{k+2}}\cdots b_{i_j} > b_{i'_1}\cdots b_{i'_{k-1}}b_{i'_{k+1}}b_{i_k}b_{i_{k+2}}\cdots b_{i_j}.
$$
Thus $x^y=x^{b_{i_1}\cdots b_{i_j}}$ with $i_1 \geq i_2 \geq \dots \geq i_j$. Equivalently we can write 
$$
x^y=x^{b_r^{e_r}\cdots b_1^{e_1}}
$$
    for some non-negative integers $e_r, \dots, e_1$.
    By Lemma \ref{lemma 3}, there exists a positive integer $e$ such that $b_i^e \in \Zen (\langle x, B \rangle)$ for all $i$. Hence we may assume that $e_i < e$ for all $i$, because if we write $x^y=x^{b_r^{e_r} \cdots b_j^{e+t} \cdots b_1^{e_1}}$, with $t$ non-negative integer, we have 
    \begin{eqnarray*}
        x^y&=&x^{b_r^{e_r} \cdots b_j^{e+t} \cdots b_1^{e_1}} \\
        &=&x^{b_r^{e_r} \cdots b_j^{e}b_j^t \cdots b_1^{e_1}} \\
        &=&x^{b_j^{e}b_r^{e_r} \cdots b_j^t \cdots b_1^{e_1}} \\
        &=&{(x^{b_j^e})}^{b_r^{e_r} \cdots b_j^t \cdots b_1^{e_1}} \\
        &=&x^{b_r^{e_r} \cdots b_j^t \cdots b_1^{e_1}}, 
    \end{eqnarray*}
    by the fact that $b_j^e \in \Zen(\langle x, B \rangle)$.
    So $|x^{v(G)}|<e^r.$ Thus $x^{v(G)}$ is finite for all $x \in G$. We conclude therefore that 
    $$
    v(G)=[w(G), \underbrace{G, \dots, G}_{m \tim}]
    $$
    is $\FC$-embedded in $G$.
\end{proof}

\section{Bounded case}

Recall that for a non-empty set $I$, a filter over $I$ is a set $\mathcal{F} \subseteq \mathcal{P}(I)$, where $\mathcal{P}(I)$ denotes the set of all subsets of $I$, satisfying the following conditions:
\begin{itemize}
\item[(i)] $\varnothing \not \in \mathcal{F}$, $I \in \mathcal{F}$;
\item[(ii)] if $X,Y \in \mathcal{F}$, then $X \cap Y \in \mathcal{F}$;
\item[(iii)] if $X \in \mathcal{F}$ and $X \subseteq Y \subseteq I$, then $Y \in \mathcal{F}$.
\end{itemize}

The filter $\mathcal{F}$ is principal if there exists a non-empty set $Y\subseteq I$ such that
$$
\mathcal{F}=\{X \subseteq I | Y \subseteq X\},
$$
and non-principal otherwise. An example of a non-principal filter over an (infinite) set $I$ is the so-called cofinite filter
$$
\mathcal{F}=\{X \subseteq I | I\smallsetminus X \mbox{ is finite}\}.
$$

A filter $\mathcal{U}$ over $I$ is called an ultrafilter if, for every $X \subseteq I$, either $X \in \mathcal{U}$ or $I \smallsetminus X \in \mathcal{U}$. This is equivalent to saying that $\mathcal{U}$ is a maximal filter over $I$. Also, $\mathcal{U}$ is a non-principal ultrafilter if and only if it contains the cofinite filter (see \cite[Proposition 1.4]{E}). Given an ultrafilter $\mathcal{U}$ over $I$ and a family $\{G_i\}_{i \in I}$ of groups, the ultraproduct modulo $\mathcal{U}$ is the quotient set of the Cartesian product $\Pi_{i \in I} G_i$ with respect to the equivalence relation defined as follows: the tuples $(g_i)_{i \in I}$ and $(h_i)_{i \in I}$ of the Cartesian product are equivalent modulo $\mathcal{U}$ if and only if $\{i \in I | g_i=h_i\} \in \mathcal{U}$. Thus the ultraproduct modulo $\mathcal{U}$ can be seen as the quotient of the unrestricted direct product of groups $G_i$ by the subgroup consisting of all tuples $(g_i)_{i \in I}$ such that $\{i \in I | g_i=1\} \in \mathcal{U}$.

Recall that the width of a group-word $w$ in a group $G$ is  the supremum of the minimum length of all decompositions of an element $g$ in $w(G)$ as a product of elements of $G_w^*$ as $g$ ranges over $w(G)$. Clearly a word $w$ has finite width at most $k \in \N$ if and only if any product of $k+1$ elements or more of $G_w^*$ can be expressed as a product of at most $k$ elements of $G_w^*$.

\begin{proposition}\label{proposizione 1}
    Let $r\geq 1$. Suppose that $w$ is a $\left(\frac{1}{m+1}\right)$-concise word, with $m \in \N$ and $G$ is a group in which $w$ takes precisely $r$ values. Then the order of 
     $$
    [w(G), \underbrace{G, \dots, G}_{m \tim}]
     $$
     is $\{m,r\}$-bounded.
\end{proposition}
\begin{proof}
    Assuming $w$ involves $n$ variables, write $w=w(x_1, \dots, x_n)$, and set 
    \begin{eqnarray*}
    v&=&v(x_1, \dots, x_n, x_{n+1}, \dots, x_{n+m}) \\
    &=&[w(x_1,\dots,x_n), x_{n+1}, \dots, x_{n+m}].
    \end{eqnarray*}
    
    Then 
    $$
    v(G)=[w(G), \underbrace{G, \dots, G}_{m \tim}].
    $$
    By way of contradiction, suppose there exists a family of groups $\mathcal{G}=\{G_i\}_{i \in \N}$ with the property that $|(G_i)_w|\leq r$ for all $i \in \N$ but
    $$
    \lim_{i \rightarrow \infty}|v(G_i)|=\infty.
    $$
    Consider a non-principal ultrafilter $\mathcal{U}$ over $\N$, and let $Q$ be the ultraproduct modulo $\mathcal{U}$ of $\mathcal{G}$. Then, by the fact that for a given integer $r$, the property that a given word takes at most $r$ values in a group can be expressed as a sentence in the first-order language of groups and \cite[Lemma 3.1]{DST}, we have $|Q_w|\leq r$, because $\N \in \mathcal{U}$. As $w$ is $\left(\frac{1}{m+1}\right)$-concise, it follows that $v(Q)$ is finite. In particular, $v$ has finite width, say $k$, in $Q$. Now the fact that for a given positive integer $k$, the property that a given word has finite width at most $k$ in a group can be expressed as a sentence in the first-order language of groups and \cite[Lemma 3.1]{DST} yield that there exist $X \in \mathcal{U}$ such that $v$ has finite width at most $k$ in $G_i$ for all $i \in X$. Hence every element of $v(G_i)$ can be written as a product of at most $k$ elements of $(G_i)_v^*$. Moreover from $|(G_i)_w|\leq r$, we get $|(G_i)_v|\leq (2r)^{2^m}$ for all $i \in \N$, by Lemma \ref{lemma 1}. Therefore, $|v(G_i)|\leq (2r)^{2^m k}$ for all $i \in X$. As noted above, $\mathcal{U}$ contains the cofinite filter over $\N$, so $X \cap Y \in \mathcal{U}$ for every cofinite subset $Y$ of $\N$. In particular, $X \cap Y$ is non-empty. Therefore every cofinite subset of $\N$ contains some element $i$ for which $|v(G_i)|\leq (2r)^{2^m k}$. This is incompatible with the assumption that $|v(G_i)|$ goes to infinity because $\lim_{i \rightarrow \infty}|v(G_i)|=\infty$ implies that for every $M \in \N$ there exist an index $\nu_M \in \N$ such that $|v(G_i)|>M$ for every $i > \nu_M$, but if we choose $M=(2r)^{2^m k}$ and consider the set $L=\{1, 2, \dots, \nu_M\}$ we have that $X=\N \smallsetminus L$ is a cofinite subset of $\N$ and therefore it contains some element $j$, which is greater than $\nu_M$, for which $|v(G_j)|\leq (2r)^{2^m k}$, a contraddiction.
\end{proof}

\begin{lemma}\label{lemma 4}
 Let $w=w(x_1, \dots, x_n)$ be a group-word and set 
    $$
    v=[w(x_1,\dots,x_n), x_{n+1}, \dots, x_{n+m}].
    $$
    If $G$ is a $\BFC (w)$-group such that $|x^{G_w}|\leq r$ for all $x \in G$, then $G$ is a $\BFC (v)$-group and $x^{G_v}$ has $\{m, n, r\}$-bounded order for all $x \in G$.
\end{lemma}
\begin{proof}
Let $x \in G$ and $y \in G_v$. By Lemma \ref{lemma 1} there exist $w_1, \dots, w_k \in G_w^*$ such that $y=w_1 \cdots w_k$, with $k\leq 2^m$. By \cite[Proposition 2.9(ii)]{DST17} $G$ is an $\BFC(w^{-1})$-group and $x^{G_{w^{-1}}}$ has $\{n, r\}$-bounded order. Note that $G_{w^{-1}}={G_w}^{-1}$, so it follows that $x^{G_w^*}$ is finite with $\{n, r\}$-bounded order. Let $x^{G_w^*}=\{x^{b_1}, \dots, x^{b_s}\}$ and put $A=\{b_1, \dots, b_s\}$. Note that $s$ is an $\{n, r\}$-bounded integer. By \cite[Lemma 2.2]{DST} we have that $x^y=x^{w_1 \cdots w_k}=x^{a_1 \cdots a_k}$, with $a_1, \dots, a_k \in A$. Therefore $|x^{G_v}|\leq |A|^{2^m}=s^{2^m}$. Therefore $x^{G_v}$ has $\{m, n, r\}$-bounded order
\end{proof}

\begin{lemma}\label{lemma 5}
    Let $w=w(x_1, \dots, x_n)$ be a $\left(\frac{1}{m+1}\right)$-concise word, with $m\in \N$ and set 
    $$
    v=[w(x_1,\dots,x_n), x_{n+1}, \dots, x_{n+m}].
    $$
    Let $G$ be an $\BFC (w)$-group such that $|x^{G_w}|\leq r$ for all $x \in G$, and let $B$ be a finite subset of $G_v^*$. Then, for any $x \in G$, there exists an $\{m,n, r, |B|\}$-bounded positive integer $e$ such that $b^e \in \Zen (\langle x, b \rangle)$ for all $b \in B$.
\end{lemma}
\begin{proof}
    Following the proof of Lemma \ref{lemma 3}, by \cite[Lemma 2.7(ii)]{DST17} the set $(J/\Zen(J))_w$ is finite of $\{m, n, r, |B|\}$-bounded order and, by Proposition \ref{proposizione 1}, the number $e$ is $\{m, n, r, |B|\}$-bounded.
\end{proof}

\begin{thmB}\label{teorema B}
  Let $w$ be a $\left(\frac{1}{m+1}\right)$-concise word, with $m \in \N$ and let $G$ be an $\BFC (w)$-group. Then 
  $$
    [w(G), \underbrace{G, \dots, G}_{m \tim}]
  $$
  is $\BFC$-embedded in $G$.
\end{thmB}
\begin{proof}
    Set $v=[w(x_1,\dots,x_n), x_{n+1}, \dots, x_{n+m}]$. We have that $G$ is a $\BFC(w)$-group, therefore it exists a positive integer $r$ such that $|x^{G_w}|\leq r$ for any $x \in G$. Then, by Lemma \ref{lemma 4}, $G$ is a $\BFC (v)$-group and $x^{G_v}$ has $\{m, n, r\}$-bounded order for any $x \in G$. Let $x \in G$ and choose $b_1, \dots, b_s \in G_v^*$ such that $x^{G_v^*}=\{x^{b_1}, \dots, x^{b_s}\}$. Write $B=\{b_1, \dots, b_s\}$. Define the order $<$ on the set of all (formal) products of the form $b_{i_1}\cdots b_{i_j}$, with $1\leq i_k \leq s$ and $j \geq 1$, as in the proof of Theorem \ref{teorema A}.

    Let $y \in v(G)$. As in the proof of Theorem \ref{teorema A}, we can write
$$
x^y=x^{b_s^{e_s}\cdots b_1^{e_1}}
$$
    for some non-negative integers $e_s, \dots, e_1$. Since $s$ is $\{m,n, r\}$-bounded by Lemma \ref{lemma 4}, it follows, from
     Lemma \ref{lemma 5}, that there exists an $\{m,n,r\}$-bounded positive integer $e$ such that $b_i^e \in \Zen (\langle x, B \rangle)$ for all $i$. Hence we may assume that $e_i < e$ for all $i$, by the same argument shown in the proof of Theorem \ref{teorema A}, and so $|x^{v(G)}|<e^s.$ Thus $x^{v(G)}$ is finite of $\{m,n,r\}$-bounded order for all $x \in G$. We conclude therefore that 
    $$
    v(G)=[w(G), \underbrace{G, \dots, G}_{m \tim}]
    $$
    is $\BFC$-embedded in $G$.
\end{proof}

\section{Examples}
Every $\frac{1}{n}$-concise word is a $\frac{1}{m}$-concise word for every $m \geq n$. The next result shows another way to obtain $\frac{1}{m}$-concise words.

\begin{proposition}
    Let $w=w(x_1, \dots, x_n)$ be a group-word, and set
    $$
    v=[w(x_1, \dots, x_n), x_{n+1}].
    $$
    If $w$ is $\frac{1}{m+1}$-concise, with $m \in \N$, then $v$ is $\frac{1}{m+1}$-concise.
\end{proposition}
\begin{proof}
Let $G$ be a group and let's assume that $G_v$ is finite. 
Since $v(G)'$ is finite (see \cite[Proposition 1]{FMT}), we may assume that $v(G)$ is abelian. It follows that every subgroup of $v(G)$ is finitely generated. Let $K=\langle k_1, \dots, k_\ell \rangle$ be a finitely generated subgroup of $G$ such that $v(G)=v(K)$. Consider now a generator of the group 
$$
[v(G), \underbrace{G, \dots, G}_{m \tim}]. 
$$
It is of the form $[y, g_1, \dots, g_m]$, with $y \in v(G)$ and $g_i \in G$. In particular it lies in the subgroup
$$
[v(G), \langle g_1 \rangle, \dots, \langle g_m \rangle].
$$ 
Put $g_i=k_{\ell +i}$ and $H=\langle k_1, \dots, k_\ell, k_{\ell +1}, \dots, k_{\ell +m} \rangle$. Note that $v(H)=v(G)$. Also we have that $|H : C_H(v(H))|$ is finite, because every $h \in H_v$ has only finitely many conjugates in $H$ (that's because $H_v$ is a normal set) and $|H : C_H(v(H))|=|H: \bigcap_{h \in H_v}C_H(h)|\leq \Pi_{h \in H_v}|H : C_H(h)|$. So from the fact that $$C_H(v(H)) \subseteq  C_H([v(H), \underbrace{H, \dots, H}_{m-1 \tim}])$$ we have that
$$
|H : C_H([v(H), \underbrace{H, \dots, H}_{m-1 \tim}])|
$$
is also finite. We claim that 
$$
|[v(H),\underbrace{H, \dots, H}_{m-1 \tim}] : [v(H), \underbrace{H, \dots, H}_{m-1 \tim}] \cap \Zen (H)|
$$
is also finite. In fact note that the set
$$
\{[x, k_i] | x \in H_w, i=1, \dots, \ell + m\}
$$
is finite because it is a subset of $G_v$, and therefore, by \cite[Lemma 4.1]{DST}, $H_w$ is contained in finitely many right cosets of $w(H) \cap \Zen (H)$. Hence $(H/w(H)\cap \Zen(H))_w$ is finite and so $(H/\Zen(H))_w$ is finite. Since $w$ is $\frac{1}{m+1}$-concise, we obtain that
$$
[w(H), \underbrace{H, \dots, H}_{m \tim}]\Zen(H)/\Zen(H) \cong [v(H),\underbrace{H, \dots, H}_{m-1 \tim}]/[v(H),\underbrace{H, \dots, H}_{m-1 \tim}]\cap \Zen(H)
$$
is finite.
From this it follows that 
$$
[[v(H),\underbrace{H, \dots, H}_{m-1 \tim}], H]=[v(H),\underbrace{H, \dots, H}_{m \tim} ]
$$
is finite by \cite[Corollary p. 103]{R}. In fact to apply the Corollary we consider the four subgroups 
$$
H, \quad [v(H),\underbrace{H, \dots, H}_{m-1 \tim}],\quad
[v(H), \underbrace{H, \dots, H}_{m-1 \tim}] \cap \Zen (H),\quad
 C_H([v(H), \underbrace{H, \dots, H}_{m-1 \tim}])
$$
of the group $H$. The subgroup 
$$
[v(H),\underbrace{H, \dots, H}_{m-1 \tim}]
$$
is normal in $H$ because it is a subgroup of $v(H)$ which is abelian and fully-invariant in $H$. The subgroup 
$$
C_H([v(H), \underbrace{H, \dots, H}_{m-1 \tim}])
$$
is normal in $H$ because it is a centralizer of a normal subgroup and 
$$
[v(H), \underbrace{H, \dots, H}_{m-1 \tim}] \cap \Zen (H)
$$
is normal in $H$ because it is an intersection of normal subgroups. Moreover 
$$
[H, [v(H), \underbrace{H, \dots, H}_{m-1 \tim}] \cap \Zen (H)]=1=
[[v(H),\underbrace{H, \dots, H}_{m-1 \tim}],C_H([v(H), \underbrace{H, \dots, H}_{m-1 \tim}])]
$$
trivially.
So we have that
$$
[v(H),\underbrace{H, \dots, H}_{m \tim}]=[v(G), \underbrace{H, \dots, H}_{m \tim}]
$$ is finite. In particular, $[v(G), \langle g_1 \rangle, \dots, \langle g_m \rangle ]$ is finite. Thus $[y, g_1, \dots, g_m]$ is periodic. Therefore
$$
[v(G), \underbrace{G, \dots, G}_{m \tim}]
$$ 
is a finitely generated periodic abelian group, and so it is finite. That proves that $v$ is $\frac{1}{m+1}$-concise.
\end{proof}

\begin{lemma}\label{esempio}
Let $G$ be a group and let $y, b \in G$ such that $b^2=1$ and $[y, y^b]=1$. Then for every $m \geq 0$ the following equality holds
$$
[y, \prescript{}{m+1}{b}]=y^{(-1)^m2^mb}y^{(-1)^{m+1}2^m}.
$$
\end{lemma}
\begin{proof}
    We argue by induction on $m$. Let $m=0$, then $[y,b]=y^{-1}byb=y^{-1}y^b=y^by^{-1}$.
    Assume the assertion true for $m\geq 0$. Then we have 
    \begin{eqnarray*}
        [y,\prescript{}{m+2}{b}]&=&[[y,\prescript{}{m+1}{b}],b] \\
        &=& [y^{(-1)^m2^mb}y^{(-1)^{m+1}2^m},b] \\
        &=& y^{(-1)^{m+2}2^m}y^{(-1)^{m+1}2^mb}by^{(-1)^{m}2^mb}y^{(-1)^{m+1}2^m}b \\
        &=&y^{(-1)^{m+2}2^m}y^{(-1)^{m+1}2^mb}y^{(-1)^{m}2^m}y^{(-1)^{m+1}2^mb} \\
        &=&y^{(-1)^{m+2}2^{m+1}}y^{(-1)^{m+1}2^{m+1}b}
    \end{eqnarray*}
    
\end{proof}

According to \cite{I}, for any odd integer $n > 10^{10}$ and any prime number $p> 5000$, the word $v(x,y)=[[x^{pn},y^{pn}]^n,y^{pn}]^n$ is not concise. Indeed, Ivanov constructed a $2$-generator torsion-free group $A$ whose center is cyclic and $A/\Zen(A)$ is infinite of exponent $p^2n$, such that $v$ takes only two values in $A$ and the nontrivial value is a generator of $\Zen(A)$. In \cite[Section 4]{BKS}, the authors considered a modification of Ivanov's example, namely the wreath product \begin{eqnarray}\label{gruppo}
    G=A \mathrm{wr} B,
\end{eqnarray} 
where $B=\langle b \rangle$ is a cyclic group of order 2. Taking \begin{eqnarray}\label{parola}
   w(x, y)=v(x^2, y^2), 
\end{eqnarray}
they showed that $|G_w|\leq 4$ and $b^{w(G)}$ is infinite. In \cite[Section 4]{DST} they showed that $b^{[w(G),G]}$ is also infinite. In a similar way we now prove that $b^{[w(G), G, \dots, G]}$, where $G$ is repeated $m$ times, is also infinite for every $m \in \N$. This implies that $w$ is not $\frac{1}{m}$-concise for every $m \in \N$.

\begin{proposition}
    There exist a group-word $w$ and a $\BFC(w)$-group $G$ such that 
    $$
    [w(G), \underbrace{G, \dots, G}_{m \tim}]
    $$
    is not $\FC$-embedded in $G$ for every $m \in \N$.
\end{proposition}
\begin{proof}
Let $G$ and $w$ be as in (\ref{gruppo}) and (\ref{parola}), respectively. Then, by \cite[Proposition 4.1]{BKS}, $G$ is a $\BFC(w)$-group. Denote by $K=A\times A^b $ the base group of $G$. For any odd integer $t\geq 1$, let $N=\langle v_0^t, (v_0^b)^t \rangle$, where $v_0 \in A$ is the nontrivial value of $v(x, y)$ in $A$. Notice that $N$ is central in $K$ and closed under conjugation by $b$ so that $N$ is a normal subgroup of $G$. Also, since $K/N$ has odd exponent $p^2tn$ and $|G/K|=2$, we have 
$$
K/N=\{g^2N | g \in G\}
$$
and consequently $(K/N)_v=(G/N)_w$. Hence $v_0N \in (G/N)_w$, and therefore $v_0^kN \in w(G/N)$ for any integer $k$. It follows that 
$$
b^{[v_0^k, \prescript{}{m}{b} ]}N \in (bN)^{[w(G/N), G/N, \dots, G/N]},
$$
where $G/N$ is repeated exactly $m$ times, with $m \in \N$.
Now $b^{[v_0^k,  \prescript{}{m}{b}]}N=b[b,[v_0^k, \prescript{}{m}{b}]]N=b[v_0^k,\prescript{}{m+1}{b}]^{-1}N$. By Lemma \ref{esempio} we have 
        \begin{eqnarray*}
            b[v_0^k,\prescript{}{m+1}{b}]^{-1}N&=&bv_0^{k(-1)^m2^m}v_0^{k(-1)^{m+1}2^mb}N \\
            &=&b(v_0^{(-1)^m2^m}v_0^{(-1)^{m+1}2^mb})^kN,
        \end{eqnarray*}
 where     $v_0^{(-1)^m2^m}v_0^{(-1)^{m+1}2^mb}N$   has order $t$ in $G/N$. Thus 
 $$
 |\{b^{[v_0^k,\prescript{}{m}{b}]}N | k \in \Z \}|=t
 $$
 and so
 $$
 |(bN)^{[w(G/N), G/N, \dots, G/N]}|\geq t.
 $$
 In particular $|b^{[w(G),G, \dots, G]}|\geq t$. Since $t$ is an arbitrary odd positive integer, we conclude that $b^{[w(G),G, \dots, G]}$ is infinite. Therefore 
 $$
    [w(G), \underbrace{G, \dots, G}_{m \tim}]
    $$
    is not $\FC$-embedded in $G$ for every $m \in \N$.
\end{proof}

\begin{corollary}
    There exist a group-word $w$ which is not $\frac{1}{m}$-concise for every $m \in \N$ and neither $0$-concise.
\end{corollary}

\section*{Acknowledgements}
The authors have been partially supported by the National Group for Algebraic and Geometric Structures,
and their Applications (GNSAGA – INdAM).

\end{document}